\documentclass[12pt]{article}
\usepackage{graphicx}
\usepackage{
amsthm,amsmath,amssymb,color
}

\newtheorem{theorem}{Theorem}[section]
\newtheorem{lemma}[theorem]{Lemma}

\newtheorem{proposition}[theorem]{Proposition}

\theoremstyle{definition}
\newtheorem{definition}[theorem]{Definition}

\theoremstyle{remark}
\newtheorem{remark}[theorem]{Remark}

\numberwithin{equation}{section}

\def\C{\mathbb C}
\def\R{\mathbb R}
\def\Z{\mathbb Z}
\def\N{\mathbb N}
\def\T{\mathbb T}
\def\Q{\mathbb Q}

\def\Lin{\mathcal L}

\def\span{\mathop{\text{span}}}

\newcommand\minus{%
  \setbox0=\hbox{-}%
  \vcenter{%
    \hrule width\wd0 height \the\fontdimen8\textfont3%
  }%
}

\title{Approximate Controllability of the Two Trapped Ions System}

\author{Esteban Paduro\thanks{
Departamento de Matem\'atica, Universidad Federico Santa Mar\'ia, Valpara\'iso, Chile,
esteban.paduro@gmail.com}
\thanks{The first author was supported in part by scholarship CONICYT Mag\'ister Nacional 2012, project FONDECYT  1120610 and project CONICYT ACT-1106.}, 
Mario Sigalotti
\thanks{
INRIA, Team GECO,
Centre de Recherche Saclay - \^Ile-de-France, France,
mario.sigalotti@inria.fr}
\thanks{The second author was partially supported by 
the European Research Council, ERC
StG 2009 GeCoMethods, contract number 239748
and  by the iCODE
institute, research project of the Idex Paris-Saclay.}
}

\begin{document}

\maketitle

\begin{abstract}
We prove the approximate controllability of a bilinear Schr\"odinger equation modelling a two trapped ions system. A new spectral decoupling technique is introduced,   which allows  to analyze 
the controllability of the infinite-dimensional system through 
 finite-dimensional considerations.
\end{abstract}

\section{Introduction}\label{ch:introduction}
In this paper we study the controllability of a system modelling a two trapped ions system driven by two laser beams. 

\subsection{The Model}
The two trapped ions system  models a pair of identical charged particles confined to a small region of the space.
Both ions are stabilized by the same spatial oscillation, given by a harmonic oscillator with of frequency $\omega$. 
The controls are monochromatic lasers of frequencies  $\Omega$ and  $\Omega\pm \omega$.

The state of the system is represented by a function $\phi=\left( \phi_{gg},\phi_{ge},\phi_{eg},\phi_{ee} \right)\in\left(L^2(\R)\right)^4$, which represents the  different possible configurations of electrons in the ions ($e$ for excited state, $g$ for ground state). In the Lamb--Dicke limit, the system can be written (see \cite{Rouchon2008}) as the following {\it two trapped ions system}  
\begin{equation}\label{S2IA_intro}
\begin{array}{rcl}
i\frac{d}{dt}\phi_{gg} & = &  (u_1 + u_{1r}a^{\dag}+ u_{1b}a) \phi_{eg} + (u_2 + u_{2r}a^{\dag}+ u_{2b}a) \phi_{ge},\\
i\frac{d}{dt}\phi_{eg} & = &  (u_1^{*} + u_{1r}^{*}a+ u_{1b}^{*}a^{\dag}) \phi_{gg} + (u_2 + u_{2r}a^\dag+ u_{2b}a) \phi_{ee},\\
i\frac{d}{dt}\phi_{ge} & = &  (u_1 + u_{1r}a^{\dag}+ u_{1b}a) \phi_{ee} + (u_2^{*} + u_{2r}^{*}a+ u_{2b}^{*}a^{\dag}) \phi_{gg},\\
i\frac{d}{dt}\phi_{ee} & = &  (u_1^{*} + u_{1r}^{*}a+ u_{1b}^{*}a^{\dag}) \phi_{ge} + (u_2^{*} + u_{2r}^{*}a+ u_{2b}^{*}a^{\dag}) \phi_{eg},\\
\end{array}
\end{equation}
where ${\protect u_1, u_{1r}, u_{1b}, u_2, u_{2r}, u_{2b} \in \C }$ are the controls of the system and $a=\frac{1}{\sqrt{2}}\left(x+\partial_x\right)$, $a^\dag=\frac{1}{\sqrt{2}}\left(x-\partial_x\right)$ are the creation and annihilation operators.

We are interested in the approximate controllability problem for system~\eqref{S2IA_intro} which can be read as follows:\\
\begin{quote}
Given $\varepsilon>0$, $\phi_0, \phi_T \in \left(L^2(\R)\right)^4$ with $\|\phi_0\|_{\left(L^2(\R)\right)^4}=\|\phi_T\|_{\left(L^2(\R)\right)^4}$,  find $T_0>0$ such that for any $T\geq T_0$ there exist bounded piecewise constant controls  $u_1(t)$, $u_{1r}(t)$, $u_{1b}(t)$, $u_2(t)$, $u_{2r}(t)$, $u_{2b}(t)$, 
with initial data $\phi(0)=\phi_0$ satisfying
    \begin{equation*}
        \left\|\phi(T) - \phi_T \right\|_{\left(L^2(\R)\right)^4} < \varepsilon.
    \end{equation*}
\end{quote}

\subsection{Existing Literature}

For a single trapped ion counterpart of \eqref{S2IA_intro}, exact controllability has been proved on some finite-dimensional subspaces of $\left(L^2(\R)\right)^2$ in \cite{Law1996} and \cite{Kneer1998}. 
In \cite{Yuan2007} an approximate controllability result on $\left(L^2(\R)\right)^2$ is obtained using spectral properties of coupling operators, allowing the decoupling of the internal dynamics. A similar approach is proposed in a more general framework in \cite{Burgarth}. 
A different approach, based on adiabatic evolution, is proposed in \cite{Adami2005}. 
The system modelling the single trapped ion before Lamb--Dicke limit has been studied in  \cite{Ervedoza2009b}, obtaining approximate controllability in $\left(L^2(\R)\right)^2$ and also in with respect to some stronger norms. Such results are based 
on 
a set of explicit estimates of the approximation error with respect to the Lamb--Dicke limit. The approximate controllability of a different model for the interaction of between a bosonic mode and a two-level system is proved in \cite{panoptikon}. 

For several trapped ions in the Lamb--Dicke limit, in \cite{Bloch2010} (see also \cite{Rangan2004a}) an approximate controllability result is obtained, 
based on  the analysis on a sequence of nested finite-dimensional system, which can be decoupled from the rest of system. In the recent paper \cite{Keyl2014} the (approximate) controllability of the system is established by considering  different families of controlled dynamics. 
Many other works studying the two trapped ions system deal with the construction of quantum gates \cite{Ions1995,Childs2000,Jonathan2000} but up to our knowledge, they do not present general controllability results for system~\eqref{S2IA_intro}.

\subsection{Main Results}

\newtheorem*{mainTh}{\bf Theorem \ref{mainTh}}

\begin{theorem}[Approximate Control of the two trapped ions system]\label{mainTh}
   Let $\varepsilon>0$, $M>0$ and $(\phi_0, \phi_T)\in \left(L^2(\R)\right)^4$, with $\|\phi_0\|=\|\phi_T\|$. Then, there exists $T_0>0$ such that for all $T>T_0$ we can find 
   $u_1$, $u_{1r}$, $u_{1b}$, $u_2$, $u_{2r}$, $u_{2b}$ $\in L^\infty((0,T),\C)$ 
   with $L^\infty$-norm smaller than $M$  such that the solution $\phi(t,x)$ of system~(\ref{S2IA_intro}) with initial data $\phi(0)=\phi_0$ satisfies
  $$\left\|\phi(T)-\phi_T\right\|_{\left(L^2(\R)\right)^4}<\varepsilon.$$
\end{theorem}

\newtheorem*{mainCor}{\bf Corollary \ref{mainCor}}

For the proof we introduce a new finite-dimensional approximation of system~\eqref{S2IA_intro} with good control properties. In a second step we study spectral properties of the coupling operators of the original system; this allows us to construct 
the control laws of the original system based on those of the finite-dimensional approximation. Finally, using estimates for the approximation error, we obtain the approximate controllability of system~\eqref{S2IA_intro}.

The proof also shows that not all vector fields are necessary for the approximate controllability of the system. In particular, it is sufficient 
to exploit red shift only (respectively, blue shift only), that is, the system is still approximately controllable if one sets to zero the controls
$u_{1b},u_{2b}$ (respectively, the controls
$u_{1r},u_{2r}$). 
See the appendix for more details on this issue.

\subsection{Outline of the Work}

The paper is organized as follows. In Section~\ref{ch:decoplingOscillatory} we develop the technical tools required for the proof. Section~\ref{ch:TTI} is devoted to the proof of Theorem \ref{mainTh}. In Section~\ref{ch:galerkinTTI} we present a modal representation for the system. Section~\ref{ch:modalTTI} discusses the modal approximation of the problem and some auxiliary controllability properties. In Section~\ref{ch:decoupledTTI} we present the decoupled modal approximation and the proof of the main theorem.

\section{A Decoupling  Technique for Oscillatory Systems}\label{ch:decoplingOscillatory}

In this section we study a decoupling technique for oscillatory systems. This technique is based on the spectral decomposition of skew-hermitian operators and 
non-resonance conditions that make the decoupling possible. This is the main technical tool used to pass from the controllability result for the finite-dimensional approximation to the approximate controllability of the infinite-dimensional system.

\subsection{Preliminary Definitions}\label{definitions}

Let $X$ be an infinite-dimensional separable Hilbert space and $\{ \phi_j \}_{j\in \N}$ an orthonormal basis of $X$. Let $U: X\to X$ be a skew-hermitian operator with compact resolvent. 
Denote by $\Sigma(U)=\{\omega_j\}_{j=1}^\Gamma \subset [0,+\infty)$ the sequence of distinct moduli of the eigenvalues of $U$, with $\Gamma\in \N\cup\{\infty\}$. We say that the elements of $\Sigma(U)$ are the {\it frequencies associated with the operator $U$}. 

For each frequency $w\in \Sigma(U)$, we define the subspaces $A_{w}\subset X$ as the space generated by eigenfunctions  associated with the eigenvalues with modulus $w$, that is,
$$A_{w}=\span\{ x \mid U x = \lambda x , |\lambda|=w \},$$
and for sets of frequencies $B\subset \Sigma(U)$ we set 
$$A_B=\bigoplus\limits_{w\in B} A_w.$$

\begin{definition}[Resonant Class]
For $\nu\neq 0$ we define the set of $\Q$-resonant frequencies with $\nu$ as
$$R(\nu)=\left\{ w\in \Sigma(U) \mid  \frac{w}{\nu}\in\Q\setminus\{0\} \right\}.$$
 We find also useful to   set $R(0)=\{0\}$ and to introduce, for every $m\in \N$ and every $\nu\in \R$, 
$$R^m(\nu)=R(\nu)\cap \{\omega_j\}_{j<m}.$$

For every  $m\in\N$, we define an equivalence relation in $\{\omega_1, \dots,\omega_{m-1}\}$ by  $\omega_j \sim \omega_k$ if and only if $R^m(\omega_j)=R^m(\omega_k)$ and we denote by $N(m)$ 
 the number of 
 equivalence classes of such a relation. 
For each class $C_j\ne \{0\}$, $j=1,\dots,N(m)$, we choose $\nu_j>0$ such that $\frac{w}{\nu_j}\in\N$ for each $w\in C_j$ ($\nu_j$ exists because $C_j$ is finite). We also set $v_j=0$ if $C_j=\{0\}$. By definition of $\sim$, the elements of $\{ \nu_j \mid j=1,\dots,N(m),\ \nu_j\ne 0\}$ are $\Q$-linearly independent.
\end{definition}

\begin{definition}[Decoupled Decomposition]
For every $m\in \N$ with $m\leq \Gamma$ the {\bf decoupled decomposition of $U$ at order $m$} 
is the family of operators $U_1, \dots, U_{N(m)}$, $U_{\mathrm{dec}}, U_\rho : X \to X$ given by 
\begin{equation}\label{descomp}
\left\{\begin{array}{r@{\hspace{1mm}}l}	
    U_j &=U \,\Pi[A_{R^m(\nu_j)}], \quad j=1,\dots,N(m),\\
	U_{\mathrm{dec}} &= U \,\Pi[A_{\omega_m}], \\
	U_\rho &= U -\sum_{j=1}^{N(m)} U_j - U_{\mathrm{dec}} = U \left(I -\sum_{j=1}^{N(m)} \Pi[A_{R^m(\nu_j)}] - \Pi[A_{\omega_m}]\right),
\end{array}\right.
\end{equation}
where $\Pi[C]$ denotes the orthogonal projection over the space $C$. By definition
\begin{equation}\label{descomp4}
    U= U_1 +\dots+U_{N(m)}+ U_{\mathrm{dec}}+U_{\rho}.
\end{equation}

\end{definition}

Some basic properties of this decomposition are summarized in the following lemma.

\begin{lemma}\label{lemmadescomp}
    Let $U$ be a skew-hermitian operator with compact resolvent and let $U_1, \dots, U_{N(m)}, U_{\mathrm{dec}}, U_\rho$ be the  decoupled decomposition at order $m$ defined as above. Then 
    \begin{itemize}
      \item[(a)] 
	  \begin{itemize} 
		\item $\mathrm{im}\left(U_j\right)\subset A_{R^m(\nu_j)}, \quad j=1,\dots,N(m)$,
		\item $\mathrm{im}\left( U_{\mathrm{dec}} \right)\subset A_{\omega_m}$,
		\item $\mathrm{im}\left( U_\rho \right)\subset$ {\scriptsize $\left(\bigoplus\limits_{j=1}^{N(m)} A_{R^m(\nu_j)} \oplus A_{\omega_m}\right)^\perp$}.
	   \end{itemize}
      \item[(b)] Each pair of operators in the decomposition commute. Moreover, if $U, V\in \{U_1, \dots, U_{N(m)}, U_{\mathrm{dec}}, U_\rho\}$ and $U\neq V$ then $U\,V=V\,U=0$.
    \end{itemize}
\end{lemma}
\begin{proof} $(a)$ It is enough to notice that the spaces $A_{R^m(\nu_i)}$ are direct sum of eigenspaces of $U$. The  space $A_{R^m(\nu_j)}$ is  then invariant under $U$. $(b)$ It follows from the fact that eigenspaces associated with different eigenvalues are orthogonal, and therefore the image of each operator is contained in the kernel of the other.
\end{proof}

\subsection{Approximate Decoupling of Skew-Hermitian Operators}
In this section we develop the main technical tool used in this paper, namely, an approximation method for the operators $\exp\left( U_j t \right)$ by means of $\exp(U \tau)$. This technique is based on the geometric fact classically known as the ``irrational windings of the torus'', which says that the integer multiples of an irrational number modulo  $1$ are dense in interval the $[0,1)$. In \cite{Childs2000,Gulde2003,Yuan2007} similar decoupling techniques are used for control purposes. Our formulation gives an abstract framework which could be applied to other problems.

For the proof we need a $n$-dimensional formulation of the mentioned result. The proof of this result is classical (see, e.g., \cite[Prop.~1.4.1]{MR1326374}).

\begin{lemma}[Irrational Winding of the Torus]\label{winding} 
Let $\T^n=\R^n/\Z^n$ be the $n$-dimensional torus with the usual distance. Let $\varphi$ the automorphism: $x\to x+\omega $ (mod 1), where $\omega \in \R^n$, $x\in \T^n$. Then the orbits of $\varphi$ are everywhere dense if and only if
$$k\cdot \omega\in\Z\text{ with }k\in \Z^n \Rightarrow k=0.$$
\end{lemma}

The main result of the section is the following.

\begin{theorem}[Approximate Decoupling of Skew-Hermitian Operators]\label{contcontrol}
Let $X$ be a separable Hilbert space and $U: X \to X$ a skew-hermitian operator with compact resolvent. Let $\Sigma(U)$, $R(\cdot)$ and $N(\cdot)$ be  defined as in Section~\ref{definitions}. Suppose 
that
 there exists $m\in\N$ such that $\omega_m\ne 0$ and $\frac{\omega_j}{\omega_m}\in (\R\setminus\Q) \cup \{0\}$ for each $j<m$.
	
Let $Y$ be a subspace of $X$ such that 
\begin{equation}\label{hypY}
\bigoplus\limits_{j<m} A_{\omega_j} \subset Y \quad \text{and}\quad \bigoplus\limits_{j> m}A_{\omega_j} \subset Y^\perp.
\end{equation}

Then, given  $\hat{t}\in \R$, $\ell \in 1,\dots, N(m)$, and $\varepsilon >0$ there exist  $\bar{t}
\in \R$ and a unitary operator $\Sigma
:X\to X$ such that $\exp(\bar{t}\, U)$ admits the decomposition
\begin{equation}\label{decomposition}
\exp(\bar{t}\, U) =\exp (\hat{t} \,U_\ell)  \Sigma = \Sigma \exp (\hat{t} \,U_\ell),
\end{equation}
where $U_\ell$ is defined by \eqref{descomp} and the operator $\Sigma$ satisfies
$$
\left\| (\Sigma - I)|_{Y} \right\|_{\Lin(Y,X)} < \varepsilon  . $$

\end{theorem}

\begin{proof}
Let $\ell \in \{1,\dots, N(m)\}$, $\hat{t}\in \R$, and $\varepsilon >0$ be fixed. We  split the proof in 4 steps.
\begin{description}
\item[Step 1 ] Consider the decoupled decomposition \eqref{descomp} of the operator $U$, with $m$ satisfying the hypotheses of the theorem. By Lemma~\ref{lemmadescomp} the terms in the decomposition commute and then we have
\begin{equation}\label{expdescomp}
    \exp(t U)= \exp(t U_1)  \cdots  \exp(t U_{N(m)}) \exp(t U_{\mathrm{dec}}) \exp(t U_\rho).
\end{equation}

\noindent
By  definition of $\nu_\ell$, either $\nu_\ell=0$ or $\frac{w}{\nu_\ell}\in\N$ with $w\in R^m(\nu_\ell)$. 
Let $\hat\nu_\ell=\nu_\ell$ if $\nu_\ell\ne 0$ while, if $\nu_\ell=0$, choose $\hat \nu_\ell=1$.
Define $\bar{t}$ by
\begin{equation}\label{bart}
\bar{t} = \hat{t} + \frac{2 \pi}{\hat \nu_\ell} s,\quad s\in \N,
\end{equation}
where $s$ will be chosen later. Then we have
\begin{equation*}
 w \bar{t} = w \hat{t} + 2\pi \frac{w}{\hat \nu_\ell} s \equiv w \hat{t}\quad (\mbox{mod}~2\pi),\qquad\forall w \in R^m( \nu_\ell).
\end{equation*}
(Notice that $w=0$ if $\nu_\ell=0$.)
By definition of $U_\ell$, it follows 
that
\begin{equation}\label{expbart}
\exp(\bar{t} U_\ell)=\exp(\hat{t} U_\ell),
\end{equation}
independently of the choice of $s\in \N$.

\item[Step 2 ] By definition of the $\Q$-resonant frequency classes, 
the elements of $\left\{\nu_j \right\}_{j\in \{1,\dots,\ell-1,\ell+1,\dots,N(m)\}, \nu_j\ne 0}\cup\{\omega_m\}$ are $\Q$-linearly independent.
Let $\hat N(m)$ be equal to $N(m)$ if $\nu_j\ne 0$ for every $j\in\{1,\dots,\ell-1,\ell+1,\dots,N(m)\}$ and to $N(m)-1$ otherwise. Define $\hat \nu_1,\dots,\nu_{\hat N(m)}$ in such a way that  $\left\{\nu_j \right\}_{j \in \{1,\dots,\ell-1,\ell+1,\dots,N(m)\}, \nu_j\ne 0}\cup\{\omega_m\}=\left\{\hat \nu_j \right\}_{j=1}^{\hat N(m)}$. 

Consider the application
$$\begin{array}{cccl}
	F: & \R & \to & \T^{\hat N(m)}
	, \\
	   & s & \mapsto & s \frac{2\pi}{\hat\nu_\ell}\left( \hat \nu_1,\dots,\hat \nu_{\hat N(m)}\right).
\end{array}$$
 By Lemma \ref{winding} we obtain that $F(\N)$ is dense in the torus $\T^{\hat N(m)}$. 
 Since $F(s)$ can be taken as close to $-\hat{t}\left( \hat \nu_1,\dots,\hat\nu_{\hat N(m)}\right)$ as desired, 
it is possible, for every $c>0$ (to be chosen later suitably small) to select $s\in\N$ in such a way that
\begin{equation*}
d_{\T^{1}}\left( s \frac{2\pi}{\hat\nu_\ell} \hat \nu_j,-\hat{t}\hat \nu_j  \right) < c \varepsilon,\qquad \mbox{for $j=1,\dots,\hat N(m)$}.
\end{equation*}
By definition of $\bar t$, the latter system of inequalities can be rewritten as 
\begin{equation*}
d_{\T^{1}}\left( \bar{t}\,\hat \nu_j,0  \right) < c \varepsilon,\qquad \mbox{for $j=1,\dots,\hat N(m)$}.
\end{equation*}
Equivalently, there exist $\delta_1,\dots,\delta_{\hat N(m)}$ such that  
\begin{equation}\label{toruswind}
|\delta_j|<c\varepsilon,\quad \delta_j= \bar{t}\,\hat \nu_j  \quad (\mbox{mod }2\pi),\qquad \mbox{for $j=1,\dots,\hat N(m)$}.
\end{equation}

\item[Step 3 ] Let $h\in \{1,\dots,\ell-1,\ell+1,\dots,N(m)\}$. 
By construction, $U_h$ only has finitely many eigenvalues, so we can write its spectral decomposition as
\begin{equation}\label{specdecomp1}
\bar{t} U_h = i \theta_1 \bar{t}\, \Pi_{\theta_1}+\dots+i \theta_n \bar{t} \,\Pi_{\theta_n},
\end{equation}
where $i\theta_1,\dots, i\theta_n\in i\R$ are the  
eigenvalues of $U_h|_{A_{R^m(\nu_h)}}$ and $\Pi_{\theta_k}$ the projection operator onto the eigenspace associated with $i\theta_k$. Using the representation (\ref{specdecomp1}) we can compute the exponential of the operator $U_h$ as follows
$$\exp(\bar{t} U_h)= e^{i \theta_1\bar{t}}\,\Pi_{\theta_1}+\dots+e^{i \theta_n\bar{t}}\,\Pi_{\theta_n} +\Pi[A_{R^m(\nu_h)}^\perp].$$

Fix $k=1,\dots,n$. 
Notice that either $\nu_h=0=\theta_k$ or there exists $j\in \{1,\dots,\hat N(m)\}$ such that $\nu_h=\hat \nu_j$ and $\frac{|\theta_k|}{\hat \nu_j}$ is in $\N$. 
By equation (\ref{toruswind}) it then follows that
$$d_{\T_1}\left(\theta_k \bar{t},0\right)=d_{\T_1}\left( \frac{\theta_k}{\hat \nu_j} \hat \nu_j \bar{t},0\right)\leq  \frac{|\theta_k|}{\hat \nu_j} c\varepsilon.$$

Hence, 
$$
  \left\|e^{\bar{t}U_h}-I\right\|_{\Lin(X)}  \leq \sum\limits_{k=1}^n\left|e^{i\bar{t}\,\theta_k}-1\right| \leq \sum\limits_{k=1}^n d_{\T^1}(\theta_k \bar{t},0) \leq c\varepsilon\sum\limits_{k=1}^n \frac{|\theta_k|}{\hat\nu_j},
$$
with $e^{\bar{t}U_h}=I$ if $\nu_h=0$. 
Moreover, reasoning similarly for 
$U_{\mathrm{dec}}$, whose only possible nonzero eigenvalues are $\pm i \omega_m$, we get
$$\left\|e^{\bar{t}U_{\mathrm{dec}}}-I\right\|_{\Lin(X)} <  2c\varepsilon.$$
Finally, 
for every $c'>0$ there exists a choice of  $\bar t$ as in \eqref{bart} such that 
\begin{equation}
    \label{estimateApprox}
  \left\|e^{\bar{t}U_h}-I\right\|_{\Lin(X)}  < c' \varepsilon,\qquad \mbox{for $h\in \{1,\dots,\ell-1,\ell+1,\dots,N(m),\mathrm{dec}\}$}. 
\end{equation}

\item[Step 4 ] 
Define
$$\Sigma=\left( \prod\limits_{j=1, j\neq\ell}^{N(m)}\exp(\bar{t}\, U_j) \right)\exp(\bar{t}\, U_{\mathrm{dec}}) \exp(\bar{t} \,U_{\rho}).$$
Then \eqref{decomposition} follows from the commutativity of the terms in the decomposition (\ref{expdescomp}) and from the identity (\ref{expbart}).

By hypothesis, $Y\subset \mathrm{Ker} \left(U_\rho\right)$ and therefore $\exp(\tau U_\rho) \phi =\phi$ for $\phi \in Y$. Then estimate \eqref{estimateApprox} with $c'=1/N(m)$ implies that $\Sigma$ satisfies
$$\left\|  (\Sigma - I)|_Y \right\|_{\Lin(Y,X)} \leq \sum_{j\neq \ell}\left\| e^{\bar{t}U_j}-I \right\|_{\Lin(Y)} + \left\| e^{\bar{t}U_{\mathrm{dec}}}-I \right\|_{\Lin(X)} < N(m) \frac{\varepsilon}{N(m)}=\varepsilon.$$
\end{description}
\end{proof}

Based on the unitarity of the evolution, we next provide an estimate for the tracking error made by iterated 
approximations. 

\begin{lemma}[Iterated Approximations]\label{errorEstimation}
  Let $X$ be a Hilbert space and $Y $ a subspace of $X$. Take $x_0\in Y$, $\|x_0\|=1$ and let $\Upsilon_1,\dots, \Upsilon_N$ be a sequence of unitary operators such that $x_k=\Upsilon_k x_{k-1} \in Y$ for $k=1,\dots,N$. Let $\widetilde{\Upsilon}_k$ be an approximation of $\Upsilon_k$ of the form $\widetilde{\Upsilon}_k= \Sigma_k \Upsilon_k$, where $\Sigma_k$ is an unitary operator on $X$ that satisfies
  $$\left\| (\Sigma_k - I)|_Y\right\|_{\Lin(Y,X)} < \varepsilon_k, \quad \text{for some }\varepsilon_k>0.$$
  Then the approximate trajectory $\tilde{x}_0=x_0$, $\tilde{x}_k=\widetilde{\Upsilon}_k \tilde{x}_{k-1}$ satisfies
\begin{equation*}
    \left\| x_n-\tilde{x}_n \right\| < \sum_{k=1}^n \varepsilon_k,\qquad \text{ for }n =1,\dots,N.
\end{equation*}

\end{lemma}

\begin{proof}
We prove the lemma by induction. For $\tilde{x}_1$, we have:
$$\tilde{x}_1=\Sigma_1 \Upsilon_1 x_0= \Upsilon_1 x_0 + (\Sigma_1-I) \Upsilon_1 x_0 = x_1 + \widehat{\delta_1}$$
where $\|\widehat{\delta_1}\|_X< \varepsilon_1$ because $\Upsilon_1 x_0 \in Y$.

Now suppose that for some $n\leq N-1$ we known that $\tilde{x}_n=x_n+\widehat{\delta}_n$, with $\|\widehat{\delta_n}\|< \sum_{k=1}^{n} \varepsilon_k$. We have to prove that the same is valid for $n+1$.
$$\tilde{x}_{n+1}=\Sigma_{n+1} \Upsilon_{n+1} x_n = \Upsilon_{n+1} x_n + (\Sigma_{n+1}-I)\Upsilon_{n+1} x_n+\Sigma_{n+1} \Upsilon_{n+1} \hat{\delta}_n = x_{n+1} + \hat{\delta}_{n+1}$$
where $\|\widehat{\delta}_{n+1}\|_X< \varepsilon_{n+1}+\|\widehat{\delta}_n\|_X < \sum_{k=1}^{n+1} \varepsilon_k$. Therefore we obtain
$$\|x_n-\tilde{x}_n\|_X<  \sum_{k=1}^n \varepsilon_k,\qquad n\leq N.$$
This concludes the proof of Lemma~\ref{errorEstimation}.
\end{proof}

\section{Control of the Two Trapped Ions System}\label{ch:TTI}

In this section we present the decoupled modal decomposition of the two trapped ions system and 
we use it to prove the approximate controllability of the latter.

\subsection{Galerkin Representation}\label{ch:galerkinTTI}

For convenience, let us rewrite here the {\bf two trapped ions system} introduced in \eqref{S2IA_intro}, that is,  
\begin{equation*}
\begin{array}{rcl}
i\frac{d}{dt}\phi_{gg} & = &  (u_1 + u_{1r}a^{\dag}+ u_{1b}a) \phi_{eg} + (u_2 + u_{2r}a^{\dag}+ u_{2b}a) \phi_{ge}  ,\\
i\frac{d}{dt}\phi_{eg} & = &  (u_1^{*} + u_{1r}^{*}a+ u_{1b}^{*}a^{\dag}) \phi_{gg} + (u_2 + u_{2r}a^\dag+ u_{2b}a) \phi_{ee}  ,\\
i\frac{d}{dt}\phi_{ge} & = &  (u_1 + u_{1r}a^{\dag}+ u_{1b}a) \phi_{ee} + (u_2^{*} + u_{2r}^{*}a+ u_{2b}^{*}a^{\dag}) \phi_{gg}  ,\\
i\frac{d}{dt}\phi_{ee} & = &  (u_1^{*} + u_{1r}^{*}a+ u_{1b}^{*}a^{\dag}) \phi_{ge} + (u_2^{*} + u_{2r}^{*}a+ u_{2b}^{*}a^{\dag}) \phi_{eg}  ,\\
\end{array}
\end{equation*}
and recall that $a=\frac{1}{\sqrt{2}}\left(x+\partial_x\right)$, $a^\dag=\frac{1}{\sqrt{2}}\left(x-\partial_x\right)$. 

Choose as 
orthonormal basis of  $\left(L^2(\R)\right)^4$ the family $\{\phi_j\}_{j=1}^{\infty}$ defined by
\begin{equation}
\phi_{4j+1}  =  \begin{pmatrix} \left| j \right\rangle \\ 0 \\ 0 \\ 0 \end{pmatrix},
\phi_{4j+2}  =  \begin{pmatrix} 0 \\ \left| j \right\rangle \\ 0 \\ 0 \end{pmatrix} ,
\phi_{4j+3}  =  \begin{pmatrix} 0 \\ 0 \\ \left| j \right\rangle \\ 0 \end{pmatrix}  ,
\phi_{4j+4}  =  \begin{pmatrix} 0 \\ 0 \\ 0 \\ \left| j \right\rangle \end{pmatrix} ,
\end{equation}
where $\left| j \right\rangle$ denotes the $j$-th Hermite function, $j=0,1,2,\dots$
The orthonormal basis identifies a system of coordinates in the space of wavefunctions $\psi\in \left(L^2(\R)\right)^4$ 
through the relation 
\begin{equation}\label{basisTTI}
{\displaystyle
\psi=\sum\limits_{n=0}^\infty\left[ c_{gg}^n \begin{pmatrix} \left| n \right\rangle \\ 0 \\ 0 \\ 0 \end{pmatrix}
+c_{eg}^n \begin{pmatrix}  0\\ \left|n \right\rangle \\0 \\0\end{pmatrix}
+c_{ge}^n \begin{pmatrix}  0\\0\\ \left|n \right\rangle \\0 \end{pmatrix}
+c_{ee}^n \begin{pmatrix}  0\\0\\0\\ \left|n \right\rangle \end{pmatrix} \right].}
\end{equation}

It is useful to split the controls in their real and imaginary parts in order to write a system with real-valued controls:
\begin{equation}\label{real-imag}
\begin{array}{lcr@{\hspace{0.5mm}}lclr@{\hspace{0.5mm}}l}
  u_1 &=& v_1&+i w_1  , & u_2 & = & v_2&+i w_2 , \\
  u_{1r} &=& v_{1r}&+i w_{1r} , & u_{2r} &=& v_{2r}&+i w_{2r} , \\
  u_{1b} &=& v_{1b}&+i w_{1b} , & u_{2b} &=& v_{2b}&+i w_{2b} .
\end{array}
\end{equation}

For $j,k\in\N$, define the skew-adjoint operators ${E}_{j,k}, {F}_{j,k}: \left(L^2(\R)\right)^4\to \left(L^2(\R)\right)^4$ by their actions on the basis $\{\phi_j\}_{j=1}^{\infty}$ as follows
\begin{equation}
\begin{array}{llll}
  {E}_{j,k}\phi_j=i \phi_k ,& {E}_{j,k}\phi_k=i \phi_j ,&{F}_{j,k}\phi_j=- \phi_k ,& {F}_{j,k}\phi_k= \phi_j ,\\
   {E}_{j,k}\phi_\ell=0 ,&  {F}_{j,k}\phi_\ell=0 ,& \mbox{ for }\ell\notin\{j,k\}.
\end{array}
\end{equation}

We can rewrite system~\eqref{S2IA_intro} as 
\begin{eqnarray}\label{GS2IA}
    \frac{d}{dt}\phi &= &\big(v_1 V_1 + w_1 W_1 + v_{1r} V_{1r}+w_{1r} W_{1r}+v_{1b} V_{1b}+w_{1b} W_{1b}\\
    \nonumber& &+v_2 V_2 + w_2 W_2 + v_{2r} V_{2r}+w_{2r} W_{2r}+v_{2b} V_{2b}+w_{2b} W_{2b}\big)\phi,
\end{eqnarray}
where 
\begin{equation}\label{couplingOperators}
\begin{array}{ccrl}
    V_1&=& -&\sum_{n=0}^{\infty} \left({E}_{4n+1,4n+2} + {E}_{4n+3,4n+4} \right) , \\
    W_1&=& &\sum_{n=0}^{\infty}\left( {F}_{4n+1,4n+2} + {F}_{4n+3,4n+4} \right) ,\\
    V_{1r}&=& -&\sum_{n=0}^{\infty}\sqrt{n+1}\left( {E}_{4n+2,4n+5} +  {E}_{4n+4,4n+7}\right)  ,\\
    W_{1r}&=& &\sum_{n=0}^{\infty}\sqrt{n+1}\left( {F}_{4n+2,4n+5} +  {F}_{4n+4,4n+7}\right) ,\\
    V_{1b}&=& -&\sum_{n=0}^{\infty}\sqrt{n+1}\left( {E}_{4n+1,4n+6} +  {E}_{4n+3,4n+8}\right)  ,\\
    W_{1b}&=& -&\sum_{n=0}^{\infty}\sqrt{n+1}\left(  {F}_{4n+1,4n+6} +  {F}_{4n+3,4n+8}\right)  ,\\
    V_2&=& -&\sum_{n=0}^{\infty} \left({E}_{4n+1,4n+3} +  {E}_{4n+2,4n+4}\right)  ,\\
    W_2&=& &\sum_{n=0}^{\infty} \left({F}_{4n+1,4n+3} +  {F}_{4n+2,4n+4} \right) ,\\
    V_{2r}&=& -&\sum_{n=0}^{\infty}\sqrt{n+1}\left( {E}_{4n+1,4n+7} +  {E}_{4n+2,4n+8}\right)  ,\\
    W_{2r}&=& &\sum_{n=0}^{\infty} \sqrt{n+1}\left( {F}_{4n+1,4n+7} +  {F}_{4n+2,4n+8}\right)  ,\\
    V_{2b}&=& -&\sum_{n=0}^{\infty}\sqrt{n+1}\left( {E}_{4n+3,4n+5} +  {E}_{4n+4,4n+6}\right)  ,\\
    W_{2b}&=& -&\sum_{n=0}^{\infty} \sqrt{n+1}\left( {F}_{4n+3,4n+5} +  {F}_{4n+4,4n+6}\right)  .
\end{array}
\end{equation}

\subsection{Control of the Modal Approximations}\label{ch:modalTTI}

Let us study the controllability of {\it modal (or Galerkin) approximations} of system~\eqref{couplingOperators}, obtained by truncating the high energy levels of the system in order to obtain a finite-dimensional reduction.

For every $n\in \N$, let ${Y_n}=\span\{ \phi_j \mid j=1,\dots,n \}\subset \left(L^2(\R)\right)^4$.
The {\bf modal approximation of order ${n}$ of the two trapped ions system} is the control system in ${Y_{4n}}$ given by
\begin{equation}\label{AM2IA}
\begin{array}{ll}
    \frac{d}{dt}\phi =\hspace{-3mm}& \big(v_1 {V}_1^{(4n)} + w_1 {W}_1^{(4n)} + v_{1r} {V}_{1r}^{(4n)}+w_{1r} {W}_{1r}^{(4n)}+v_{1b} {V}_{1b}^{(4n)}+w_{1b} {W}_{1b}^{(4n)}\\
    & +v_2 {V}_2^{(4n)} + w_2 {W}_2^{(4n)} + v_{2r} {V}_{2r}^{(4n)}+w_{2r} {W}_{2r}^{(4n)}+v_{2b} {V}_{2b}^{(4n)}+w_{2b} {W}_{2b}^{(4n)}\big)\phi,
\end{array}
\end{equation}
with 
coupling operators 
defined by 
\begin{equation}\label{Galerkin}
Z_{\gamma}^{(4n)}=\Pi[{Y}_{4n}]Z_{\gamma},\quad Z_{\gamma\star}^{(4n)}=\Pi[{Y}_{4n}]Z_{\gamma\star},\qquad Z=V,W,\ \gamma=1,2,\ \star=b,r.
\end{equation}
By a useful abuse of notation, the operators $Z_{\gamma}^{(4n)},Z_{\gamma\star}^{(4n)}$ are identified in \eqref{AM2IA} with operators in $\Lin({Y_{4n}})$, while equation~\eqref{Galerkin} actually defines operators in $\Lin((L^2(\R))^4)$.

Up to the permutation in the coordinates of ${Y_{4n}}$ 
defined by 
\begin{eqnarray}
\lefteqn{P(c_{gg}^0,c_{eg}^0,c_{ge}^0,c_{ee}^0,\dots,c_{gg}^{n-1},c_{eg}^{n-1},c_{ge}^{n-1},c_{ee}^{n-1})=}\nonumber\\
&\left(c_{gg}^0,\dots,c_{gg}^{n-1},c_{eg}^0,\dots,c_{eg}^{n-1},c_{ge}^0,\dots,c_{ge}^{n-1}, c_{ee}^0,\dots,c_{ee}^{n-1}\right),\label{permutation}
\end{eqnarray}
 the coupling operators have the following matrix representations 
{\small
\begin{equation}\label{couplingMatricesTTI}
\begin{array}{@{\hspace{-1mm}}r@{\hspace{4mm}}l@{\hspace{1mm}}r@{\hspace{1mm}}l@{\hspace{1mm}}c}
P {V}_1^{(4n)} P^{-1}  = & - i
\left(
\begin{array}{c@{\hspace{0.5mm}}c|c@{\hspace{0.5mm}}c}
  & I_n& & \\
  I_n& & & \\
  \hline
  & & &I_n \\
  & & I_n& \\
\end{array}
\right)  , &
P {W}_1^{(4n)} P^{-1}  = & \left(
\begin{array}{c@{\hspace{0.5mm}}c|c@{\hspace{0.5mm}}c}
  &I_n & & \\
  \minus I_n& & & \\
  \hline
  & & &I_n \\
  & & \minus I_n& \\
\end{array}
\right)  , \\
P {V}_{1r}^{(4n)} P^{-1} = & - i\left(
\begin{array}{c@{\hspace{0.5mm}}c|c@{\hspace{0.5mm}}c}
  &D^T & & \\
  D& & & \\
  \hline
  & & &D^T \\
  & &D & \\
\end{array}
\right)  , &
P {W}_{1r}^{(4n)} P^{-1}  = & \left(
\begin{array}{c@{\hspace{0.5mm}}c|c@{\hspace{0.5mm}}c}
  & D^T & & \\
  \minus D& & & \\
  \hline
  & & & D^T \\
  & & \minus D & \\
\end{array}
\right) , \\
P {V}_{1b}^{(4n)} P^{-1}  = & -i \left(
\begin{array}{c@{\hspace{0.5mm}}c|c@{\hspace{0.5mm}}c}
  &D & & \\
  D^T& & & \\
  \hline
  & & &D \\
  & & D^T& \\
\end{array}
\right)  , &
P {W}_{1b}^{(4n)} P^{-1}  = & \left(
\begin{array}{c@{\hspace{0.5mm}}c|c@{\hspace{0.5mm}}c}
  & D& & \\
  \minus D^T& & & \\
  \hline
  & & &D \\
  & & \minus D^T& \\
\end{array}
\right) , \\
P {V}_2^{(4n)} P^{-1}  = & -i
\left(
\begin{array}{c@{\hspace{0.5mm}}c|c@{\hspace{0.5mm}}c}
  & &I_n & \\
  & & &I_n \\
  \hline
  I_n& & & \\
  & I_n& & \\
\end{array}
\right)  ,&
P {W}_2^{(4n)} P^{-1}  = & \left(
\begin{array}{c@{\hspace{0.5mm}}c|c@{\hspace{0.5mm}}c}
  & & I_n& \\
  & & &I_n \\
  \hline
  \minus I_n& & & \\
  &\minus I_n & & \\
\end{array}
\right) , \\
P {V}_{2r}^{(4n)} P^{-1}  = & -i\left(
\begin{array}{c@{\hspace{0.5mm}}c|c@{\hspace{0.5mm}}c}
  & & D& \\
  & & & D\\
  \hline
  D^T& & & \\
  & D^T& & \\
\end{array}
\right)  , &
P {W}_{2r}^{(4n)} P^{-1}  = & \left(
\begin{array}{c@{\hspace{0.5mm}}c|c@{\hspace{0.5mm}}c}
  & & D& \\
  & & &D \\
  \hline
  \minus D^T& & & \\
  & \minus D^T& & \\
\end{array}
\right)  , \\
P {V}_{2b}^{(4n)} P^{-1}  = & -i\left(
\begin{array}{c@{\hspace{0.5mm}}c|c@{\hspace{0.5mm}}c}
  & & D^T & \\
  & & & D^T\\
  \hline
  D& & & \\
  &D & & \\
\end{array}
\right)  , &
P {W}_{2b}^{(4n)} P^{-1}  = & \left(
\begin{array}{c@{\hspace{0.5mm}}c|c@{\hspace{0.5mm}}c}
  & & D^T& \\
  & & &D^T \\
  \hline
  \minus D& & & \\
  &\minus D & & \\
\end{array}
\right)  ,
\end{array}
\end{equation}
}
where the matrix $D$ is defined as
\begin{equation}\label{defD}
D=\left(
\begin{array}{ccccc}
    0 & \sqrt{1} & & & \\
     & 0 & \sqrt{2} & & \\
    &  & \ddots & \ddots & \\
    & & & 0 & \sqrt{n-1} \\
    & & &  & 0 \\
\end{array}
\right).
\end{equation}

\begin{proposition}[Exact Controllability of the Modal Approximations of the Two Trapped Ions System]\label{modcont2IA}
    Let $n\geq 3$, $M>0$ and $(\phi_0,\phi_T) \in {Y_{4n}}\times {Y_{4n}}$, with $\|\phi_0\|=\|\phi_T\|$. Then, there exists $T_0>0$ such that for any $T\geq T_0$ we can find 
    controls in $L^\infty\left((0,T),\C\right)$ with $L^\infty$-norm smaller than $M$  
    such that the solution of system~(\ref{AM2IA}) with initial data $\phi(0)=\phi_0$ satisfies $\phi(T)=\phi_T$.
\end{proposition}

\begin{remark}
The boundedness on the controls implies that the infimum of the controllability time is positive (at it happens for the single ion case, see  \cite{Law1996}). 
\end{remark}

The proof is based on the classical Chow--Rashevskii theorem, recalled here below.  (For a proof, see, e.g., \cite{jurdjevic1996}.) Recall that, given a set $\mathcal {F}$ of smooth vector fields on a manifold $M$,  $\mathrm{Lie}_x\mathcal{F}$ denotes the evaluation at a point $x\in M$ of the Lie algebra $\mathrm{Lie}\mathcal{F}$ generated by $\mathcal{F}$. 
\begin{theorem}\label{teo3jurdjevic}
  Assume that a driftless control affine system $\dot x=\sum_{j=1}^m u_j X_j(x)$, $(u_1,\dots,u_m)\in U\subset \R^m$,  defined on a finite-dimensional compact manifold $M$ satisfies $\mathrm{Lie}_x\{X_1,\dots,X_m\}$ $=T_x M$ for all $x$ in $M$ and that the convex hull of $U$ contains the origin in its interior. Then there exists $T_0>0$ such that for any $T\geq T_0$ and any open connected set $\Omega$ in $M$, any two points of $\Omega$ can be connected by  a trajectory $x:[0,T]\to M$ of the control system with support  in $\Omega$.
\end{theorem}

The proof of Proposition~\ref{modcont2IA}, which can be found in the appendix, 
consists then in the computation of the iterated Lie brackets of the vector fields appearing in system~(\ref{AM2IA}). 
We show in the appendix, moreover, that it is possible to set some of the controls appearing in \eqref{AM2IA} identically equal to zero and still recover the controllability of the modal approximation (for every $n\geq 3$).

\begin{figure}
\begin{center}
\includegraphics[width=0.45\linewidth]{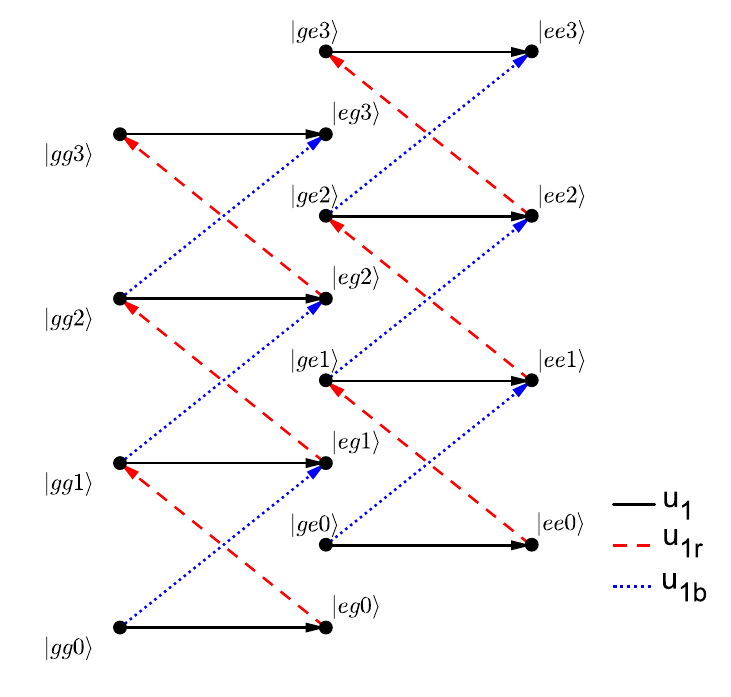}
\includegraphics[width=0.45\linewidth]{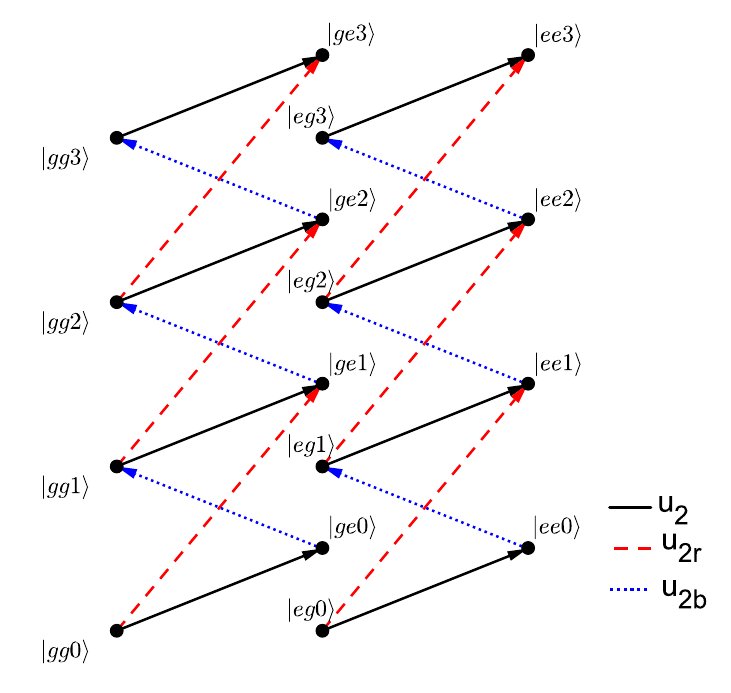}\\
\end{center}
\caption{Transition diagram between energy levels of the system by considering a finite number of Fock states in the oscillator. The lines of the same type represent the action of the same control}\label{ion2}
\end{figure}

\subsection{The Decoupled Modal Approximation}\label{ch:decoupledTTI}

The inconvenient of using the modal approximation \eqref{AM2IA} to study system~\eqref{S2IA_intro} is that whenever we apply the controls associated with the blue lasers ($u_{1b}$ and $u_{2b}$) or the red lasers ($u_{1r}$ and $u_{2r}$), some population is transmitted from the phonon level $\left|n\right\rangle$ to the phonon level $\left|n+1\right\rangle$ (see Figure~\ref{ion2}). The truncation defining the modal approximation does not keep track of this transfer, hence the control laws computed 
on the model approximations do not work for system~\eqref{S2IA_intro}. 
To overcome this issue we use a different truncation,  the {\bf decoupled modal approximation} of system~\eqref{S2IA_intro}, which depends on the spectrum of the coupling operators. 
Let us compare this approximation with the one proposed in \cite{metalemma}: the goal of both approximations is to guarantee 
that the admissible motions of the approximate system are also admissible, in an approximate way, in the full system. However, the technical 
arguments leading to the two approximations are different: in \cite{metalemma} (see also \cite{Boscain2012,Chambrion-periodic,Chambrion2009}) the approximation is based on the dynamical and spectral analysis of the drift Hamiltonian (exploiting the non-resonances of its spectrum), while in the case considered here the drift is null and the non-resonances are exploited separately for different values of the control parameters and then composed together.

The following proposition lists the main spectral properties of the operators defined  in (\ref{couplingOperators}).
\begin{proposition}[Spectral Properties of the Coupling Operators]\label{specprop1}
    Let $U$ be one of the operators $V_{1}, W_{1}, V_{2}, W_{2}$. Then the spectrum of $U$ satisfies
    \begin{itemize}
      \item[a)] the eigenvalues of $U$ are $i$ and $-i$;
      \item[b)] for each $n\in\N$ the space ${Y_{4 n}}$ is invariant under $U$. 
    \end{itemize}
    If $U$   is one of the operators $V_{1r}, W_{1r}, V_{1b}, W_{1b}, V_{2r},W_{2r}, V_{2b},W_{2b}$ 
    then the spectrum of $U$ satisfies
    \begin{itemize}
      \item[a$'$)] the eigenvalues of $U$ are $\{\pm i \sqrt{j}\}_{j=0}^\infty$;
      \item[b$'$)] $\span\{x\mid Ux=\lambda x, |\lambda|< \sqrt{n}\}\subset {Y_{4n}}$, $n\in\N$;
      \item[c$'$)] $\span\{x\mid Ux=\lambda x, |\lambda|> \sqrt{n}\}\subset { Y_{4n}}^\perp$, $n\in \N$.
    \end{itemize}

\end{proposition}

The proof follows directly from the definition of the operators given in (\ref{couplingOperators}), noticing that if $U$ is one of the operators $V_{1}, W_{1}, V_{2}, W_{2}$ then each space 
$\span\{\phi_j\mid j=4n+1,\dots,4(n+1)\}$, $n\in \N$,  is invariant under $U$, while if 
$U$   is one of the operators $V_{1r}, W_{1r}, V_{1b}, W_{1b}, V_{2r},W_{2r}, V_{2b},W_{2b}$ then 
$$\span\{x\mid Ux=\lambda x, |\lambda|= \sqrt{n}\}\subset \span\{\phi_j\mid j=4(n-1)+1,\dots,4(n+1)\}$$
for every $n\in\N$.

We apply in the following the construction of Section~\ref{definitions} to the operators $V_{1b},W_{1b},V_{1r},W_{1r},V_{2b},W_{2b},V_{2r},W_{2r}$.
Let $m \in\N$ and write $\omega_j = \sqrt{j-1}$ for $j\geq 1$. 
Let us associate with $\{\omega_j\}_{j=1}^\infty$ the integer $N(m)$ and  the frequencies $\nu_j$, $i=1,\dots,N(m)$, as detailed  in Section~\ref{definitions}.

Given $n\in\N$, set $m=n+1$ and define the {\bf decoupled modal approximation of order $n$ of the two trapped ions system} as the control system in 
${Y_{4n}}$ 
with $8 {N(m)} + 4$ controls
given by
\begin{equation}\label{AMDS2IA}
\begin{array}{ccl}
    \frac{d}{dt}\phi & =&  \Big(v_1 {V}_1^{(4n)} + w_1 {W}_1^{(4n)} +v_2 {V}_2^{(4n)} + w_2 {W}_2^{(4n)}  \\
        & &+\sum\limits_{j=1}^{N(m)}  \Big(v_{1r,j} {V}_{1r,j}^{(4n)}+w_{1r,j} {W}_{1r,j}^{(4n)}+v_{1b,j} {V}_{1b,j}^{(4n)}+w_{1b,j} {W}_{1b,j}^{(4n)}\\
        & & 
        +\ v_{2r,j} {V}_{2r,j}^{(4n)}+w_{2r,j} {W}_{2r,j}^{(4n)}+v_{2b,j} {V}_{2b,j}^{(4n)}+w_{2b,j} {W}_{2b,j}^{(4n)}\Big) \Big) \phi,
\end{array}
\end{equation}
where $Z_{\gamma}^{(4n)}$, for $Z=V,W$ and $\gamma=1,2$, are defined as in \eqref{Galerkin} and  
\begin{equation}\label{couplingOperators_DMA}
Z_{\gamma\star,j}^{(4n)}=
Z_{\gamma\star}\Pi[A_{R^m(\nu_j)}],\qquad Z=V,W,\ \gamma=1,2,\ \star=b,r.
\end{equation}

Notice that, by Proposition~\ref{specprop1} point b$'$), since $\sqrt{n}=\omega_{n+1}=\omega_m$, the operators 
$Z_{\gamma\star,j}^{(4n)}$ are indeed well-defined operators in $\Lin({Y_{4n}})$.  By Proposition~\ref{specprop1} point c$'$), moreover,
\begin{equation}\label{telescopic}
Z_{\gamma\star}^{(4n)}=\sum_{j=1}^{N(m)}Z_{\gamma\star,j}^{(4n)},\qquad Z=V,W,\ \gamma=1,2,\ \star=b,r.
\end{equation}

\begin{proposition}[Exact Controllability of the Decoupled Modal Approximation]\label{controlDecoupled}
    Let $M>0$ and $\left(\phi_0,\phi_T\right)\in \C^{4n}\times\C^{4n}$, $n\in \N$, 
    with $\|\phi_0\|=\|\phi_T\|$. Then, there exists $T_0>0$ such that for any $T\geq T_0$ we can find controls 
    $z_{\gamma},z_{\gamma\star,j}\in L^\infty\left((0,T),\R\right)$, $z=v,w$, $\gamma=1,2$, $\star=b,r$, $j=1,\dots,N(n+1)$, with $L^\infty$-norm smaller than $M$ such that the corresponding solution of system~(\ref{AMDS2IA}) with initial data $\phi(0)=\phi_0$ satisfies $\phi(T)=\phi_T$. 
\end{proposition}

The proposition follows easily from Proposition \ref{modcont2IA}, since 
the set of controlled operators in \eqref{AM2IA} is  contained in the one in \eqref{AMDS2IA}, as it follows from \eqref{telescopic}.

\begin{remark}\label{meno-controlli}
Clearly, it is also true that, if 
$$ \mathcal{U}\subset\{z_{\gamma},z_{\gamma\star}\mid z=v,w,\; \gamma=1,2,\; \star=b,r\}$$
is a family of controls which makes \eqref{AM2IA} controllable for every $n\geq 3$ (i.e., \eqref{AM2IA}  is controllable if we set to zero all the controls which are not in $\mathcal{U}$) then the corresponding decoupled modal approximation 
(i.e., system~\eqref{AMDS2IA} where all $z_{\gamma}$ which are not in $\mathcal{U}$ are set to zero, together with all $z_{\gamma\star,j}$ such that $z_{\gamma\star}$ is not in $\mathcal{U}$)
is also controllable. 
\end{remark}

\begin{remark}\label{PC}
It is well-know for finite-dimensional control systems that 
the controls in \eqref{controlDecoupled} can be taken piecewise constant. Moreover, applying Chow--Rashevskii theorem (see Theorem~\ref{teo3jurdjevic}) to the control system having as admissible vector fields $\pm Z_\gamma^{(4n)}$ and $\pm {Z}_{\gamma\star,j}^{(4n)}$, with $Z=V,W$, $\gamma=1,2$, $j=1,\dots,N(p)$, and $\star=b,r$, one deduces that $\phi_0$ can be steered to $\phi_T$ by the concatenation of the flows 
of such vector fields. Equivalently said, \eqref{controlDecoupled} is controllable by piecewise constants controls 
such that at each time instant 
at most one of them is nonzero. 
\end{remark}

\subsection{Main Theorem}\label{sec:approxControl}

\begin{theorem}[Approximate Controllability of the Two Trapped Ions System]
   Let $\varepsilon>0$, $M>0$ and $\phi_0, \phi_T\in \left(L^2(\R)\right)^4$, with $\|\phi_0\|=\|\phi_T\|$. Then, there exists $T_0>0$ such that for all $T>T_0$ we can find piecewise constant controls $u_1$, $u_{1r}$, $u_{1b}$, $u_2$, $u_{2r}$, $u_{2b}$ $\in L^\infty\left((0,T),\C\right)$, with norm smaller than $M$ 
   such that the solution 
   of system~\eqref{S2IA_intro} with initial data $\phi(0)=\phi_0$ satisfies
  $$\left\|\phi(T)-\phi_T\right\|_{\left(L^2(\R)\right)^4}< \varepsilon.$$
\end{theorem}

\begin{proof}
First notice that it is enough to prove the theorem for $\phi_0, \phi_T\in {Y_{4n}}$ for any given $n\in\ N$. 
Indeed, assume that $n$ is large enough so that there exist $\bar{\phi}_0, \bar{\phi}_T\in {Y_{4n}}$ with
$$\|\phi_0\|=\|\bar{\phi}_0\|=\|\bar{\phi}_T\|=\|\phi_T\|, \quad \|\bar{\phi}_0-\phi_0\|_{\left(L^2(\R)\right)^4}< \frac{\varepsilon}{3},\quad \|\bar{\phi}_T-\phi_T\|_{\left(L^2(\R)\right)^4}< \frac{\varepsilon}{3}.$$
Assuming that the theorem is true in the case of initial and final data in ${Y_{4n}}$ implies that there exists an admissible trajectory $\phi(\cdot)$ of~\eqref{S2IA_intro} such that 
$$\phi(0)=\bar{\phi}_0, \qquad \|\phi(T)-\bar{\phi}_T\|_{\left(L^2(\R)\right)^4}<\frac{\varepsilon}{3}.$$
The conclusion then follows from the unitarity of the flow of \eqref{S2IA_intro}.

Assume then that $\phi_0, \phi_T\in {Y_{4n}}$, $n\in \N$. 
Let $p\geq n$ be a prime number and define $Y={Y_{4p}}$. Consider the decoupled modal approximation~\eqref{AMDS2IA} of order $p$. By Proposition~\ref{controlDecoupled} we know that 
it 
is controllable.
Following Remark~\ref{PC}, 
there exists a sequence of times $t_1,\dots,t_N >0 $,  amplitudes $u_1,\dots,u_N\in [-M,M]$, and coupling operators $S_1,\dots,S_N\in \{{Z}_\gamma^{(4p)}, Z_{\gamma\star,j}^{(4p)}\mid Z=V,W,\; \gamma=1,2,\; j=1,\dots,N(p+1),\; \star=b,r\}$ such that
$$\phi_T=e^{t_N u_N S_N}\cdots e^{t_1 u_1 S_1}\phi_0.$$

In order to mimic the control scheme of the finite-dimensional system in system~\eqref{S2IA_intro}, we are going to replace the controls of the decoupled modal approximation by those given by Theorem~\ref{contcontrol} and then study how close  the final state $\phi(T)$ is to the target state $\phi_T$ .

For a given $S_k$, $k=1,\dots,N$, define 
$$\Upsilon_k=e^{t_k u_k S_k}$$
and consider the following two cases. 
\begin{itemize}
\item Either $S_k$ is of the type ${Z}_\gamma^{(4p)}$, with $Z\in\{V,W\},\; \gamma\in\{1,2\}$.
Then we drive 
the original system~\eqref{GS2IA} by the  corresponding coupling operator $Z_\gamma$ with the same control $u_k$ and for the same time $t_k$ (i.e., we set the corresponding control $z_\gamma$ at the value $u_k$, all the other controls to zero, and we follow the flow of the system for a time $t_k$).

We are then approximating the flow $\Upsilon_k$
 by an admissible flow 
 $$\widetilde{\Upsilon}_k=e^{t_k u_k Z_\gamma}$$ of \eqref{GS2IA} 
which can be written as $ \Sigma_k \Upsilon_k$ with $\Sigma_k$ unitary and $\Sigma_k|_Y=I$. 
The latter fact follows from Proposition~\ref{specprop1}, point b).

\item Or $S_k$ is of the type ${Z}_{\gamma\star,j}^{(4p)}$, with $Z\in\{V,W\}$, $\gamma\in\{1,2\}$, $\star\in\{b,r\}$, $j\in\{1,\dots,N(p+1)\}$. 
Then we apply Theorem \ref{contcontrol} with $m=p+1$, 
$U={Z}_{\gamma\star}$, 
$\ell=j$, $\hat{t}=u_kt_k$, and taking as $\varepsilon$ the quantity $\varepsilon/N$. 
Let us check that the hypotheses of the theorem are satisfied.  
On the one hand, condition \eqref{hypY} on $Y$ is verified thanks to Proposition \ref{specprop1}, points b$'$) and c$'$) (taking $n=p$). 
On the other hand, in order to check 
that $\frac{\omega_h}{\omega_m}\in (\R\setminus\Q) \cup \{0\}$ for each $h<m$, 
let us assume by contradiction that there exists $1<h<p$ such that $\frac{\sqrt{h}}{\sqrt{p}}=\frac{a}{b}$ with $a,b\in\N$ relatively prime. Therefore, 
 $p=h\frac{b^2}{a^2}$, which  implies that $a^2$ divides to $h$. Then we can write $h=s a^2$, with $s\in\N$, and  we have $p=s b^2$. However, as $p$ is prime, this implies that $b=1$ and $s=p$. Hence, $h=p a^2$. Since $h\leq p-1$, however, this leads to a contradiction. The hypotheses of Theorem~\ref{contcontrol} are then satisfied.

As a consequence of Theorem~\ref{contcontrol} we deduce that there exists 
$\bar{t}\in \R$ such that
$$\exp(\bar{t} {Z}_{\gamma\star}) = \Sigma_k \exp( u_kt_k {Z}_{\gamma\star} 
\Pi[A_{R^{p+1}(\nu_j)}])= \Sigma_k \Upsilon_k,$$
where 
 $\Sigma_k$ is a unitary operator that satisfies
$$\big\| (\Sigma_k - I)|_Y \big\|_{\Lin(Y,(L^2(\R))^4)} < \frac{\varepsilon}{N}.$$
The strategy is then to set in  the original system~\eqref{GS2IA} 
the control $z_{\gamma\star}$ corresponding to 
$Z_{\gamma\star}$ at the value $u_k'=\pm u_k$, all other controls to zero, and let the system flow for a time $\tau_k>0$, where the choice of $\tau_k$ and of the sign of $u_k'$ are such that $u_k'\tau_k=\hat t$. Hence 
$$\widetilde{\Upsilon}_k=\exp(\bar{t} {Z}_{\gamma\star})$$
is an admissible flow for system~\eqref{GS2IA}.

\end{itemize}

With the previous procedure we have obtained a control strategy for the two trapped ions system. Finally, applying Lemma \ref{errorEstimation} with $X:=\left(L^2(\R)\right)^4$, $x_0=\phi_0$
$x_j= \Upsilon_k x_{j-i}$, we obtain
$$\|\phi(T)-\phi_T\|<  \sum_{j=1}^N \frac{\varepsilon}{N} = \varepsilon.$$
Therefore system~\eqref{S2IA_intro} is approximately controllable.
\end{proof}

\begin{remark}\label{meno-controlli-conclusion}
It follows from the argument used to prove the theorem and from Remark~\ref{meno-controlli}
that if 
$$ \mathcal{U}\subset\{z_{\gamma},z_{\gamma\star}\mid z=v,w,\; \gamma=1,2,\; \star=b,r\}$$
is a family of controls which makes \eqref{AM2IA} controllable for every $n\geq 3$ 
then the same controls are also sufficient to approximately control system~\eqref{GS2IA} (or, equivalently, system~\eqref{S2IA_intro} up to identification of real and imaginary parts as in \eqref{real-imag}).
The appendix discusses which conditions on the set $\mathcal{U}$ guarantee such a controllability property.
\end{remark}

\section{Conclusions}

In this work we have studied a new decomposition method for Schr\"odinger systems based on spectral techniques. The proposed decomposition allows to obtain approximate controllability results by using finite-dimensional techniques. The method provides satisfactory theoretical results (sufficient conditions for approximate controllability), although it requires large times for the decoupling procedure. 

The method is applied to the two trapped ions model, for which we obtain  a new approximate controllability result. Since the underlying approximate controllability result is based on constructive considerations, the result on the two trapped ions model actually provides a  motion planning algorithm (as detailed in~\cite{Paduro2013}).

\subsection*{Acknowledgments}
The authors would like to thank Eduardo Cerpa and Alberto Mercado for many fruitful discussions.

\bibliographystyle{amsplain}
\bibliography{biblio}

\section*{Appendix: Controllability of modal approximations
}\label{appendixB}\label{A-B} 

The scope of this appendix is to present the Lie algebra computations necessary for the proof of Proposition~\ref{modcont2IA} and, more generally, to identify subfamilies of controls using which it is possible to control the modal approximation~(\ref{AM2IA}) for every $n$ sufficiently large.

As 
recalled above, the key criterion allowing the controllability analysis of 
(\ref{AM2IA}) 
is the Chow--Rashevskii  theorem (Theorem~\ref{teo3jurdjevic}). 
A useful approach is to apply such a criterion to the lift of system~(\ref{AM2IA}) in SU$(4n)$ and to exploit the structure of 
homogeneous space of $S^{8n-1}$ (the unit sphere in $\C^{4n}$) with respect to $SU(4n)$. Based on such lifting procedure, Albertini and D'Alessandro proved in \cite{albertini} a result implying that (\ref{AM2IA}) is controllable if and only if the Lie algebra generated by 
$\{Z_\gamma,Z_{\gamma\star}\mid Z=V,W,\;\gamma=1,2,\;\star=r,b\}$ is equal to $\mathfrak{su}(4n)$ or contains a subalgebra conjugate to $\mathfrak{sp}(2n)$. 

The main result of the appendix, which implies  Proposition~\ref{modcont2IA} as a particular case, is the following. 

\begin{proposition}\label{contacci}
Let $n\geq 3$ and  assume that $\mathcal{F}\subset \{Z_\gamma^{(4n)},Z_{\gamma\star}^{(4n)}\mid Z=V,W,\;\gamma=1,2,\;\star=r,b\}$ is such that,
for every $\gamma=1,2$ there exists  $\star\in\{r,b\}$ (possibly depending on $\gamma$)  such that 
$\{Z_\gamma^{(4n)},Z_{\gamma\star}^{(4n)}\mid Z=V,W\}\subset \mathcal{F}$.
Then the Lie algebra generated by $\mathcal{F}$ is equal to $\mathfrak{su}(4n)$.
\end{proposition}

A preliminary step in the proof of Proposition~\ref{contacci} is the analysis of the controllability of the 
{\it modal approximation of order $n$ of the Law--Eberly system}, namely, the control system in $\C^{2n}$
\begin{equation}\label{AMS1IA}
    \frac{d}{dt}\phi =\left( v V^{(2n)}+ w W^{(2n)}+v_r V_r^{(2n)} +w_r W_r^{(2n)}+v_b V_b^{(2n)} +w_b W_b^{(2n)} \right) \phi,
\end{equation}
where the controls $v, w, v, w_r, v_b, w_b,$ are real-valued and the coupling operators are defined by
\begin{equation}\label{coupmatrices}
\begin{array}{c@{\hspace{1mm}}c@{\hspace{1mm}}cc@{\hspace{1mm}}c@{\hspace{1mm}}c}
    V^{(2n)} &=&  -i\left( \begin{array}{cc} 0 &  I_n \\  I_n & 0\\   \end{array}\right), &
    W^{(2n)} &=&  \left( \begin{array}{cc} 0 &  -I_n \\  I_n & 0\\   \end{array}\right), \\
    & & & & &  \\
    V_b^{(2n)} &=&  -i\left( \begin{array}{cc} 0 &  D \\  D^T & 0\\   \end{array}\right), &
    W_b^{(2n)} &=&  \left( \begin{array}{cc} 0 &  -D \\  D^T & 0\\   \end{array}\right), \\
    & & & & &  \\
    V_r^{(2n)} &=&  -i\left( \begin{array}{cc} 0 &  D^T \\  D & 0\\   \end{array}\right), &
    W_r^{(2n)} &=&  \left( \begin{array}{cc} 0 &  -D^T \\  D & 0\\   \end{array}\right), \\
\end{array}
\end{equation}
where the matrix $D$ is defined as in \eqref{defD}.

\begin{proposition}\label{contacciEL}
Let $n\geq 2$ and  assume that $\mathcal{F}_{\mathrm{EL}}= \{Z^{(2n)},Z_{\star}^{(2n)}\mid Z=V,W\}$ for $\star=r$ or $\star=b$.
Then the Lie algebra generated by $\mathcal{F}_{\mathrm{EL}}$ is equal to $\mathfrak{su}(2n)$.
\end{proposition}
\begin{proof}
The proof is contained in \cite{Yuan2007} in the case $\star=r$.  
The case $\star=b$ can be treated in complete analogy, since by a simple reordering of coordinates one can 
transform $V_b^{(2n)}$ and $W_b^{(2n)}$ into 
\begin{equation}\label{coupmatrices2}
\begin{array}{c@{\hspace{1mm}}c@{\hspace{1mm}}cc@{\hspace{1mm}}c@{\hspace{1mm}}c}
\hat V_b^{(2n)} &=&  -i\left( \begin{array}{cc} 0 &  \hat D^T \\  \hat D & 0\\   \end{array}\right), &
   \hat  W_b^{(2n)} &=&  \left( \begin{array}{cc} 0 &  -\hat D^T \\  \hat D & 0\\   \end{array}\right), 
    \end{array}
\end{equation}
respectively, with 
$$\hat D= \left(
\begin{array}{ccccc}
    0 & \sqrt{n-1} & & & \\
     & 0 & \sqrt{n-2} & & \\
    &  & \ddots & \ddots & \\
    & & & 0 & \sqrt{1} \\
    & & &  & 0 \\
\end{array}
\right),$$
while preserving $V^{(2n)}$ and $W^{(2n)}$. 
The same arguments as in  \cite{Yuan2007} then allow to conclude.
\end{proof}

\noindent {\it Proof of Proposition~\ref{contacci}.}
First of all let us fix the following notation: in order to ease the reading of the proof, we add here below an index to each square matrix indicating its size, so that $A^{(k)}$ denotes a general $k\times k$ matrix and $0^{(k)}$ the $k\times k$ null matrix. 
We also write $\mathrm{diag}(A^{(k_1)},\dots,A^{(k_r)})$ to denote the square matrix of size $k_1+\dots +k_r$ having $A^{(k_1)},\dots,A^{(k_r)}$ as block-diagonal terms. 

Let us apply  
Proposition~\ref{contacciEL} to each of the two subsystems of \eqref{AM2IA} obtained by setting to zero
either all controls of the type $z_1, z_{1\star}$ or all controls of the type $z_2,z_{2\star}$. 
In each of the two cases,  
 we recover two decoupled copies of the Law--Eberly modal approximation of order $n$, controlled simultaneously by the same controls.
Proposition~\ref{contacciEL} then  implies 
that 
for any choice of 
$$A^{(2n)}, \left( \begin{array}{cc} B_{11}^{(n)}& B_{12}^{(n)} \\ B_{21}^{(n)} & B_{22}^{(n)}   \end{array}\right) 
\in\mathfrak{su}(2n)$$
the Lie algebra
$\mathrm{Lie}\mathcal{F}$ contains both
\begin{equation}\label{f-t}
    \left( \begin{array}{cc} A^{(2n)}& 0^{(2n)} \\ 0^{(2n)} & A^{(2n)}   \end{array}\right)\quad \mbox{and}\quad  \left( \begin{array}{cccc} B_{11}^{(n)}& 0^{(n)}& B_{12}^{(n)}&0^{(n)} \\ 0^{(n)} & B_{11}^{(n)}&0^{(n)}&B_{12}^{(n)}\\ B_{21}^{(n)}&0^{(n)}& B_{22}^{(n)}&0^{(n)}\\ 0^{(n)}&B_{21}^{(n)}&0^{(n)}&B_{22}^{(n)}   \end{array}\right).
    \end{equation}

Special types of matrix appearing in \eqref{f-t} are
\[\mathrm{diag}( A^{(n)},B^{(n)},A^{(n)},B^{(n)}),\quad \mathrm{diag}( C^{(n)},C^{(n)},D^{(n)},D^{(n)}),\]
for $A^{(n)},B^{(n)},C^{(n)},D^{(n)}\in \mathfrak{su}(n)$.
Taking brackets between them (with $A^{(n)}=0^{(n)}$ or $B^{(n)}=0^{(n)}$ and $C^{(n)}=0^{(n)}$ or $D^{(n)}=0^{(n)}$) 
and since 
$[\mathfrak{su}(n),\mathfrak{su}(n)]=\mathfrak{su}(n)$, 
one deduces that, 
\begin{equation}\label{bl-d}
\mathrm{diag}( A^{(n)},B^{(n)},C^{(n)},D^{(n)})\in \mathrm{Lie}\mathcal{F}
\end{equation}
for every $A^{(n)},B^{(n)},C^{(n)},D^{(n)}\in \mathfrak{su}(n)$.

Another special type of matrices appearing in \eqref{f-t} are 
\begin{equation}\label{f-ttt}
   \left( \begin{array}{cccc} 0^{(n)}& A^{(n)}& 0^{(n)}&0^{(n)} \\ -(A^{(n)})^\dag & 0^{(n)}&0^{(n)}&0^{(n)}\\0^{(n)}&0^{(n)}& 0^{(n)}&A^{(n)}\\ 0^{(n)}&0^{(n)}&-(A^{(n)})^\dag&0^{(n)}   \end{array}\right)
    \end{equation}
with $A^{(n)}\in \mathfrak{gl}(n)$. Taking brackets of matrices in \eqref{bl-d} and in \eqref{f-ttt}, and exploiting the fact that 
$\mathfrak{su}(n)\mathfrak{gl}(n)=\mathfrak{gl}(n)$, we conclude that all matrices of the type
\begin{equation}\label{f-ttt-t}
   \left( \begin{array}{cccc} 0^{(n)}& A^{(n)}& 0^{(n)}&0^{(n)} \\ -(A^{(n)})^\dag & 0^{(n)}&0^{(n)}&0^{(n)}\\0^{(n)}&0^{(n)}& 0^{(n)}&B^{(n)}\\ 0^{(n)}&0^{(n)}&-(B^{(n)})^\dag&0^{(n)}   \end{array}\right)
    \end{equation}
with $A^{(n)},B^{(n)}\in \mathfrak{gl}(n)$ are in $\mathrm{Lie}\mathcal{F}$.
Taking linear combinations of matrices of the type \eqref{f-ttt-t}  with those of the type \eqref{bl-d}, it turns out that $\mathrm{Lie}\mathcal{F}$ contains all matrices of the type 
$$
\mathrm{diag}( A^{(2n)},B^{(2n)})\in \mathrm{Lie}\mathcal{F},\quad A^{(2n)},B^{(2n)}\in \mathfrak{su}(2n).
$$

We are left to prove that  $\left( \begin{array}{cc} 0^{2n}& C^{2n} \\ -(C^{2n})^\dag & 0^{2n}   \end{array}\right)$ is in 
$\mathrm{Lie}\mathcal{F}$ for $C^{2n}\in \mathfrak{gl}(2n)$.
Because of \eqref{f-t} and \eqref{bl-d},  $\mathrm{Lie}\mathcal{F}$ contains all brackets 
between matrices of the type $\mathrm{diag}(A^{(n)},B^{(n)},0^{(2n)})$ and
\[ 
\left( \begin{array}{cccc} 0^{(n)}& 0^{(n)}& C^{(n)}&0^{(n)} \\ 0^{(n)} & 0^{(n)}&0^{(n)}&C^{(n)}\\-(C^{(n)})^\dag&0^{(n)}& 0^{(n)}&0^{(n)}\\ 0^{(n)}&-(C^{(n)})^\dag&0^{(n)}&0^{(n)}   \end{array}\right),\]
for $A^{(n)},B^{(n)}\in \mathfrak{su}(n)$, and $C^{(n)}\in \mathfrak{gl}(n)$.
Hence 
\begin{equation}\label{enough}
\left( \begin{array}{cccc} 0^{(n)}& 0^{(n)}& A^{(n)}&0^{(n)} \\ 0^{(n)} & 0^{(n)}&0^{(n)}&B^{(n)}\\-(A^{(n)})^\dag&0^{(n)}& 0^{(n)}&0^{(n)}\\ 0^{(n)}&-(B^{(n)})^\dag&0^{(n)}&0^{(n)}   \end{array}\right)\in \mathrm{Lie}\mathcal{F},\qquad \mbox{for all } A^{(n)},B^{(n)}\in \mathfrak{gl}(n).
\end{equation}

Taking brackets between matrices of the type \eqref{f-ttt-t} and \eqref{enough}, we 
easily deduce that 
$$\left( \begin{array}{cccc} 0^{(n)}& 0^{(n)}& 0^{(n)}&C^{(n)} \\ 0^{(n)} & 0^{(n)}&D^{(n)}&0^{(n)}\\0^{(n)}&-(D^{(n)})^\dag& 0^{(n)}&0^{(n)}\\ -(C^{(n)})^\dag&0^{(n)}&0^{(n)}&0^{(n)}   \end{array}\right)\in \mathrm{Lie}\mathcal{F},\qquad \mbox{for every }C^{(n)},D^{(n)}\in \mathfrak{gl}(n).$$
This, together with \eqref{enough}, completes the proof of the equality $ \mathrm{Lie}\mathcal{F}=\mathfrak{su}(4n)$. \hfill $\Box$

\end{document}